\documentclass[12pt]{amsart}
\usepackage[colorlinks=true,urlcolor=blue, citecolor=red,linkcolor=blue,linktocpage,pdfpagelabels, bookmarksnumbered,bookmarksopen]{hyperref}
\usepackage[hyperpageref]{backref}
\usepackage{cleveref}
\usepackage{amsthm} %for citing inside theorem header
\usepackage{latexsym,amsmath,amssymb}
\usepackage{accents}
\usepackage[colorinlistoftodos,prependcaption,textsize=tiny]{todonotes}
\usepackage{a4wide}
\usepackage{soul}
\usepackage{mathtools} %For \chi with a much lower index (\mychi)
\usepackage{xparse} %For a new definiton of \chi with a slightly lower index
  
\title[$W^{s,\frac{n}{s}}$-maps with positive distributional Jacobians]{$W^{s,\frac{n}{s}}$-maps with positive distributional Jacobians}
\author{Siran Li}
 
\address[Siran Li]{Department of Mathematics, Rice University, MS 136 P.O. Box 1892, Houston,
Texas, 77251-1892, USA \newline Department of Mathematics, McGill University, Burnside Hall, 805 Sherbrooke Street West, Montreal, Quebec, H3A 0B9, Canada.}
\email{Siran.Li@rice.edu}

\author{Armin Schikorra}
\address[Armin Schikorra]{Department of Mathematics,
University of Pittsburgh,
301 Thackeray Hall,
Pittsburgh, PA 15260, USA}
\email{armin@pitt.edu}

plus 6pt minus 12pt
\def\eps{\varepsilon}

\def\N{{\mathbb N}}

\def\S{{\mathbb S}}

\newtheorem{theorem}{Theorem}
\newtheorem{lemma}[theorem]{Lemma}

\newtheorem{proposition}[theorem]{Proposition}

\newtheorem{definition}[theorem]{Definition}
\newtheorem{example}[theorem]{Example}

\def\osc{\mathop{\rm osc\,}}
\def\diam{{\rm diam\,}}
\def\dist{{\rm dist\,}}

\def\curl{{\rm curl\,}}
\def\lip{{\rm Lip\,}}

\newcommand{\R}{\mathbb{R}}

\newcommand{\brac}[1]{\left (#1 \right )}

\newcommand{\Ep}{\bigwedge\nolimits}

\newcommand{\barint}{
\rule[.036in]{.12in}{.009in}\kern-.16in \displaystyle\int }

\newcommand{\barcal}{\text{$ \rule[.036in]{.11in}{.007in}\kern-.128in\int $}}

\def\mvint_#1{\mathchoice
          {\mathop{\vrule width 6pt height 3 pt depth -2.5pt
                  \kern -8pt \intop}\nolimits_{\kern -3pt #1}}%
          {\mathop{\vrule width 5pt height 3 pt depth -2.6pt
                  \kern -6pt \intop}\nolimits_{#1}}%
          {\mathop{\vrule width 5pt height 3 pt depth -2.6pt
                  \kern -6pt \intop}\nolimits_{#1}}%
          {\mathop{\vrule width 5pt height 3 pt depth -2.6pt
                  \kern -6pt \intop}\nolimits_{#1}}}

\numberwithin{theorem}{section} \numberwithin{equation}{section}

\newcommand{\aleq}{\precsim}

\newcommand{\aeq}{\approx}

\let\latexchi\chi
\makeatletter
\renewcommand\chi{\@ifnextchar_\sub@chi\latexchi}
\newcommand{\sub@chi}[2]{% #1 is _, #2 is the subscript
  \@ifnextchar^{\subsup@chi{#2}}{\latexchi^{}_{#2}}%
}
\newcommand{\subsup@chi}[3]{% #1 is the subscript, #2 is ^, #3 is the superscript
  \latexchi_{#1}^{#3}%
}
\makeatother

\newcommand{\Jac}{\operatorname{Jac}}
\begin{document}

\begin{abstract}
We extend the well-known result that any $f \in W^{1,n}(\Omega,\R^n)$, $\Omega \subset \R^n$ with strictly positive Jacobian is actually continuous: it is also true for fractional Sobolev spaces $W^{s,\frac{n}{s}}(\Omega)$ for any $s \geq \frac{n}{n+1}$, where the sign condition on the Jacobian is understood in a distributional sense.

Along the way we also obtain extensions to fractional Sobolev spaces $W^{s,\frac{n}{s}}$ of the degree estimates known for $W^{1,n}$-maps with positive or non-negative Jacobian, such as the sense-preserving property.
\end{abstract}

\maketitle
\tableofcontents
\sloppy

\section{Introduction}
The following well-known theorem  was first proven by Gold\v{s}te\u{\i}n and Vodopyanov \cite{VG77}; see also \cite{S88,FG95,HM02} and the recent extension to manifolds in \cite{GHP17}:
\begin{theorem}\label{th:old}
Let $\Omega \subset \R^n$ be an open set. If $f \in W^{1,n}(\Omega,\R^n)$ and 
\[\Jac(f) := \det (Df) > 0 \quad  \text{a.e. in $\Omega$},\] 
then $f$ is continuous%
% and open
.
\end{theorem}
The strict inequality $\Jac(f) > 0$ is necessary as the following counterexample shows:
\begin{example}\label{ex:counter}
Let $B$ denote the unit ball in $\R^n$. Let $\tilde{f} \in W^{1,n}(B,\R)$ be discontinuous, e.g. $\tilde{f}(x) := \log \log \frac{2}{|x|}$. Set 
\[
 f(x) := (\tilde{f}(x),0,\ldots,0).
\]
Clearly $f \in W^{1,n}(B,\R^n)$ and $\Jac(f) = \det(Df) \equiv 0$. However, $f$ is still discontinuous.
\end{example}
The aim of this note is to give a reasonable extension to Theorem~\ref{th:old} to fractional Sobolev spaces $W^{s,p}(\Omega,\R^n)$, $s \in (0,1)$. These are the spaces of maps $f \in L^p(\Omega,\R^n)$ with finite $W^{s,p}$-Gagliardo semi-norm
\[
 [f]_{W^{s,p}(\Omega)} := \brac{\int_{\Omega}\int_{\Omega} \frac{|f(x)-f(y)|^p}{|x-y|^{n+sp}}\, dx\, dy}^{\frac{1}{p}} < \infty.
\]
Clearly, for a pointwise definition of the Jacobian of $f$ to make sense, $f$ should be almost everywhere differentiable; however, as a distributional operator, the Jacobian also exists for maps in fractional Sobolev spaces $W^{s,p}$, where $s < 1$ is large enough. For the sake of presentation we restrict our attention to the critical scaling, that is to the Sobolev spaces $W^{s,\frac{n}{s}}$, $s \in (0,1)$. The space $W^{s,p}_0(\Omega)$ denotes, as usual, the closure of $C_c^\infty(\Omega)$-functions in the $W^{s,p}$-norm.
\begin{lemma}\label{la:weakjac}
Let $\Omega \subset \R^n$ be open with smooth boundary, $n \geq 2$. For\footnote{The case $s =1$ is also true (with $W^{0,\infty}$ replaced by $BMO$): it is the famous theorem by Coifman-Lions-Meyer-Semmes \cite{CLMS}.} $s \in (\frac{n-1}{n},1)$ and $f \in W^{s,\frac{n}{s}}(\Omega)$ the Jacobian operator extends to a bounded linear operator on $W^{(1-s)n,\frac{1}{1-s}}_0(\Omega)$ in the following sense. The operator
\[
\Jac(f)[\varphi]:=\lim_{k \to \infty} \int_{\Omega} \det(Df_k)\, \varphi_k
\]
is well-defined for any $f_k \in C^\infty(\overline{\Omega})$ which is a smooth approximation of $f \in W^{s,\frac{n}{s}}(\Omega)$ and any $\varphi_k \in C_c^\infty(\Omega)$ which is a smooth approximation of $\varphi$ in $W^{(1-s)n,\frac{1}{1-s}}(\Omega)$.
\end{lemma}
We recall a proof of Lemma~\ref{la:weakjac} in \Cref{s:fracsobest}.

We will restrict our attention to the case $s \geq \frac{n}{n+1}$. This threshold appears in several situations on degree-type estimates in fractional Sobolev spaces; see, e.g., \cite{GO19,SVS19}. It is exactly the case when (up to the boundary data) a map $f \in W^{s,\frac{n}{s}}$ can serve as a testfunction for its own Jacobian $\Jac(f)$. \Cref{la:weakjac} warrants the following definition for a distributional Jacobian.
\begin{definition}\label{def:jac}
Assume $s \geq \frac{n}{n+1}$ and $\Omega \subset \R^n$ is a smooth, bounded domain. Let $f \in W^{s,\frac{n}{s}}(\Omega,\R^n)$. 

\begin{itemize}
\item We say $\Jac(f) \geq 0$ in $\Omega$ if for any $\varphi \in W_0^{s,\frac{n}{s}}(\Omega)$, $\varphi \geq 0$ a.e., there holds 
\[
\Jac(f)[\varphi] \geq 0.
\]
\item We say $\Jac(f) > 0$ if $\Jac(f) \geq 0$ and for any $\varphi \in W_0^{s,\frac{n}{s}}(\Omega)$, $\varphi \geq 0$ a.e., 
\[
\Jac(f)[\varphi] = 0 \quad \text{implies that $\varphi \equiv 0$}.
\]
\end{itemize}
\end{definition}

Our main result is the following version of Theorem~\ref{th:old} for fractional Sobolev spaces $W^{s,\frac{n}{s}}$.
\begin{theorem}\label{th:main}
Let $f \in W^{s,\frac{n}{s}}(\Omega,\R^n)$, $s \geq \frac{n}{n+1}$, for some open and bounded set $\Omega \subset \R^n$ with smooth boundary. 

If $\Jac(f) > 0$ then $f$ is continuous. 
% Moreover, for any open set $U \subset \subset \Omega$ also $f(U)$ is open.
\end{theorem}
By the counterexample, \Cref{ex:counter}, there is no hope of getting Theorem~\ref{th:main} under merely the assumption $\Jac(f) \geq 0$. However, as it is used for the planar Monge-Amp\`ere equation, a curl-free condition is a remedy -- similar properties are known, {\it e.g.} for $W^{1,n}$-maps, see \cite[Lemma 2.1.]{P04}, or $C^{0,\alpha}$, $\alpha > \frac{2}{3}$, see \cite{LP17}. Namely we have
\begin{theorem}\label{th:main2}
Let $f = (f_1,f_2) \in W^{s,\frac{2}{s}}(\Omega,\R^2)$ for some open and bounded set $\Omega \subset \R^2$ with smooth boundary and for some $s \geq \frac{2}{3}$. If $\Jac(f) \geq 0$ and if $\curl(f) = 0$ in distributional sense, i.e. if
\[
 \curl(f)[\varphi] = -\int_{\Omega} f_2\, \partial_1\varphi  - f_1\, \partial_2 \varphi= 0 \quad \mbox{for all $\varphi \in C_c^\infty(\Omega)$},
\]
then $f$ is continuous.
\end{theorem}
Along the way of proving Theorem~\ref{th:main} and Theorem~\ref{th:main2} we obtain degree estimates for maps with signed Jacobian which are of independent interest.

If $f \in W^{s,\frac{n}{s}}(\Omega)$ then for any $x_0 \in \Omega$ we have $f \in W^{s,\frac{n}{s}}(\partial B(x_0,r)) \hookrightarrow C^{0,\frac{s}{n}}(\partial B(x_0,r))$ for almost every $0 < r < \dist(x_0,\partial \Omega)$, by means of Sobolev embedding and Fubini's theorem, Lemma~\ref{la:restriction}. In particular, for any $p \in \R^n \backslash f(\partial B(x,r))$ the degree $\deg(f,B(x,r),p)$ is well-defined as the Brouwer degree of the map $\frac{f-p}{|f-p|}: \partial B_r(x) \to \S^{n-1}$ for almost every $r$, {\it cf.}~ \cite{FG95book}.

We first observe that $f$ with non-negative Jacobian is monotone in the following sense:
\begin{proposition}\label{pr:degest}
Let $f \in W^{s,\frac{n}{s}}(\Omega,\R^n)$ for $\Omega \subset \R^n$ open and $s \geq \frac{n}{n+1}$. Let $B(x,r) \subset B(x,R) \subset \Omega$ and assume that $f$ (in the trace sense) restricted to $\partial B(x,r)$ and $\partial B(x,R)$ is continuous.

If $\Jac(f) \geq 0$ in $\Omega$, then for any $p \not \in (f(\partial B(x,r)) \cup f(\partial B(x,R)))$ we have
\[
 \deg (f,B(x,r),p) \leq \deg (f,B(x,R),p).
\]
\end{proposition}
We also have
\begin{proposition}\label{pr:degest2}
Let $f \in W^{s,\frac{n}{s}}(\Omega,\R^n)$ for $\Omega \subset \R^n$ open and $s \geq \frac{n}{n+1}$. Let $B(x,R) \subset \Omega$ and assume that $f$ is continuous on $\partial B(x,R)$.

If $\Jac(f) \geq 0$, then for any $p \not \in f(\partial B(x,R))$ we have
\[
 \deg (f,B(x,R),p) \geq 0.
\]
\end{proposition}

Next, we obtain that if $f$ is continuous and the Jacobian of $f$ is positive then $f$ is sense-preserving:
\begin{proposition}\label{pr:degeststrict}
Let $f \in W^{s,\frac{n}{s}}\cap C^0(\Omega,\R^n)$, $\Omega \subset \R^n$ open, $s \geq \frac{n}{n+1}$.

If $\Jac(f) > 0$ in $\Omega$ then for any ball $B(r) \subset \Omega$ if $p \in f(B(r)) \backslash f(\partial B(r))$ then $\deg(f,B(r),p) \geq 1$.
\end{proposition}

If the Jacobian is positive, the image of a ball $f(B(r))$ has an essential diameter comparable to the diameter of $f(\partial B(r))$. This will be the main ingredient towards the proof of Theorem~\ref{th:main}.
\begin{proposition}\label{pr:diamest}
There exists some $\Lambda > 0$  depending only on the dimension such that the following holds. Let $f \in W^{s,\frac{n}{s}}(\Omega,\R^n)$, $\Omega \subset \R^n$ open, and $s \geq \frac{n}{n+1}$. Assume that $\Jac(f) > 0$ in $\Omega$. For any $B(r) \subset \subset \Omega$ such that $f\Big |_{\partial B(r)}$ is continuous, we can find a ball $B(q,R) \subset \R^n$ with 
\[
 R \leq \Lambda\ \diam(f(\partial B(r))),
\]
and 
\[
 \{x \in B(r): f(x) \not \in B(q,R)\} \quad \text{is a null set}. 
\]
The number $2R$ may be viewed as the ``essential diameter'' of $f(B(r))$.
\end{proposition}
The remainder of this paper is organized as follows. 
In Section~\ref{s:fracsobest} we refer to some needed results for Sobolev spaces.
In Section~\ref{s:signedjac} we prove the degree estimates for maps with signed Jacobians, namely Propositions~\ref{pr:degest}, \ref{pr:degest2}, \ref{pr:degeststrict}, \ref{pr:diamest}.
In Section~\ref{s:mainproof} we prove Theorem~\ref{th:main} and Theorem~\ref{th:main2}.

\textbf{Acknowledgments.}
This work has been done during SL's stay as a CRM--ISM postdoctoral fellow at Centre de Recherches Math\'{e}matiques, Universit\'{e} de Montr\'{e}al and Institut des Sciences Math\'{e}matiques. SL thanks these institutes for their hospitality. AS acknowledges funding by the Simons foundation, grant no 579261.

The authors would like to thank P. Haj\l{}asz for helpful discussions; in particular he told us about \Cref{ex:counter}.
\section{Fractional Sobolev spaces}\label{s:fracsobest}
Lemma~\ref{la:weakjac} was (essentially) proven in \cite{WY99a} as an extension of the ground-breaking paper \cite{CLMS}, which showed that Jacobians of $W^{1,n}$-maps can be tested with BMO-maps. The proof in \cite{WY99a} uses Littlewood-Paley theory and paraproducts. In \cite{Brezis-Nguyen-2011} Brezis and Nguyen gave a simpler and more elegant proof of this result for $s =\frac{n}{n+1}$. We present here the following slight adaptation of their argument due to \cite{LS18}.

We restrict our attention to the {\em a priori} estimates, from which the claim follows easily due to multi-linearity.
\begin{proof}[Proof of \Cref{la:weakjac} (a priori estimates)]
Let $\varphi \in C_c^\infty(\Omega)$ and $f \in C^\infty(\overline{\Omega})$. $\Omega$ is an extension domain, \cite{JW78,Z15}, so we may assume that $f \in W^{s,\frac{n}{s}}(\R^n) \cap C^1(\R^n)$. Then
\[
 \int_{\Omega} \det(Df) \varphi = \int_{\R^n} \det(Df) \varphi. 
\]
Extend $f$ and $\varphi$ harmonically to $\R^{n+1}_+$, say to $F$ and $\Phi$ respectively. We write $(x,t) \in \R^n \times \R_+ = \R^{n+1}_+$. By Stokes' theorem and H\"older's inequality,
\begin{align*}
 \left |\int_{\R^n} \det(Df) \varphi \right |&= \left |\int_{\R^{n+1}_+} \det(DF|D\Phi) \right |\\&\leq \brac{\int_{\R^{n+1}_+} |t^{1-\frac{s}{n}-s} DF|^{\frac{n}{s}}}^{s} \brac{\int_{\R^{n+1}_+} |t^{1-(1-s)-(1-s)n}D\Phi|^{\frac{1}{1-s}}}^{1-s}.
\end{align*}
If $s \in (\frac{n-1}{n},1]$,  then $(1-s)n \in (0,1)$. Then, by trace estimates, see {\it e.g.} \cite[Proposition~10.2]{LS18},  we have 
\[
 \brac{\int_{\R^{n+1}_+} |t^{1-\frac{s}{n}-s} DF|^{\frac{n}{s}}}^{s} \aeq [f]_{W^{s,\frac{n}{s}}(\Omega)}^n
\]
and
\[
 \brac{\int_{\R^{n+1}_+} |t^{1-(1-s)-(1-s)n}D\Phi|^{\frac{1}{1-s}}}^{1-s} \aeq [\varphi]_{W^{(1-s)n,1-s}(\Omega)}.
\]
Here we also used the fact that $[f]_{W^{s,\frac{n}{s}}(\R^n)} \aleq [f]_{W^{s,\frac{n}{s}}(\Omega)}$. This is because $\Omega$ is an extension domain; see \cite{JW78,Z15}.
We conclude, because we have shown
\[
 \int_{\Omega} \det(Df) \varphi  \aleq [f]_{W^{s,\frac{n}{s}}(\Omega)}^n\, [\varphi]_{W^{(1-s)n,1-s}(\Omega)}.
\]
\end{proof}
The ensuing result on trace operators will be useful for the subsequent developments. For detailed treatments we refer to  \cite[\textsection 2.4.2, Theorem 1]{RS96}, \cite[Theorem 7.43, Remark 7.45]{AF03} and  \cite[Lemma 36.1]{T07}.
\begin{lemma}[Trace Theorem]\label{la:trace}
Let $\Omega \subset \R^n$ be either bounded or the complement of a bounded set, with smooth boundary. If $s \in (0,1)$, $p \in (1,\infty)$ with $s-\frac{1}{p} > 0$, then the trace operator on $T = \Big |_{\partial \Omega}$ is a bounded, linear, surjective operator from $W^{s,p}(\Omega)$ to $W^{s-\frac{1}{p},p}(\partial \Omega)$. The harmonic extension is a bounded linear right-inverse of $T$.
\end{lemma}

The following is well-known for Sobolev functions in $W^{1,p}$ (it is essentially Fubini's theorem):
\begin{lemma}[Restriction theorem]\label{la:restriction}
For $\Omega$ a smooth, bounded domain let $f \in W^{s,p}(\Omega)$. Fix $x_0 \in \Omega$. There exists a representative of $f$ such that for $\mathcal{L}^1$-almost every $r \in (0,\dist(x_0,\partial \Omega))$ we have $f \in W^{s,p}(\partial B(x_0, r))$.

Moreover, for $\Omega = B(x_0,R)$ we have 
\[
 \brac{\int_0^R [f]_{W^{s,p}(\partial B(x_0,r))}^{p} dr}^{\frac{1}{p}} \aleq [f]_{W^{s,p}(B(x_0,R))}.
\]
\end{lemma}
\begin{proof}
As $\Omega$ is an extension domain, see \cite{JW78,Z15}, we may assume that $\Omega = \R^n$ and $f \in W^{s,p}(\R^n)$ with $f \equiv 0$ outside a compact set. Denote by $F: \R^{n+1}_+ \to \R$ the harmonic extension of $f$, and w.l.o.g. set $x_0=0$. Then (see \cite[Proposition 10.2]{LS18}) 
\[
 \left \|(x_{n+1})^{1-\frac{1}{p}-s} DF \right \|_{L^p(\R^{n+1}_+)} \aeq [f]_{W^{s,p}(\R^n)} < \infty.
\]
By Fubini's theorem, for $\mathcal{L}^1$-almost every $r > 0$,
\[
  \left \|(x_{n+1})^{1-\frac{1}{p}-s} DF \right \|_{L^p(\partial B(r) \times (0,\infty))} < \infty.
\]
This implies that $f \in W^{s,p}(\partial B(r))$ for almost every $r > 0$.

The last claim also follows from Fubini's theorem in $\R^{n+1}_+$:
\[
 \int_0^R [f]_{W^{s,p}(\partial B(x_0,r))}^{p} dr \aleq  \int_{0}^R \int_{\partial B(x_0,r) \times (0,\infty)} |(x_{n+1})^{1-\frac{1}{p}-s} DF|^p = \int_{B(x_0,R) \times (0,\infty)} |(x_{n+1})^{1-\frac{1}{p}-s} DF|^p. 
\]
\end{proof}

\begin{lemma}\label{la:glue}
Let $\Omega \subset \R^n$ be a bounded domain with smooth boundary. For $s \in (0,1)$, $p \in (1,\infty)$ such that $s-\frac{1}{p} > 0$:
\begin{enumerate}
 \item If $f \in W^{s,p}(\Omega)$ and $g \in W^{s,p}(\R^n \backslash \Omega)$ with $f =g$ on $\partial \Omega$ in the trace sense. Then
 \[
  h := \begin{cases}
        f \quad &\text{in $\Omega$}\\
        g \quad &\text{in $\R^n \backslash \Omega$}
       \end{cases}
 \]
belongs to the Sobolev space and
\[
 [h]_{W^{s,p}(\R^n)} \leq C(\Omega)\,\Big( [f]_{W^{s,p}(\Omega)} + [g]_{W^{s,p}(\R^n\setminus\Omega)}\Big).
\]
\item In particular, if $f \in W^{s,p}(\Omega)$ satisfies $f = 0$ on $\partial \Omega$ in the trace sense, then $f\in W^{s,p}_0(\Omega)$ and that 
 \[
  h := \begin{cases}
        f \quad &\text{in $\Omega$}\\
        0 \quad &\text{in $\R^n \backslash \Omega$}
       \end{cases}
 \]
belongs to $W^{s,p}(\R^n)$.
\end{enumerate}

\end{lemma}

\begin{lemma}\label{la:approx}
Let $B(R)$ be a ball in $\R^n$. Let $f \in W^{s,\frac{n}{s}}(B(R))$ for some $s \in (0,1)$ and $f\Big|_{\partial B(R)} \in C^0(\partial B(R))$. Then there exists an approximation $f_k \in C^\infty_c (\R^n)$ converging to $f$ in $W^{s,\frac{n}{s}}(B(R))$ and $f_k \rightrightarrows f$ uniformly on $\partial B(R)$.
\end{lemma}
\begin{proof}
W.l.o.g. $B(R) = B := B(0,1)$.

By the trace theorem, \Cref{la:trace}, $f \in W^{s-\frac{n}{s},\frac{n}{s}}(\partial B)$. Let $g$ be the harmonic extension of $f$ to $\R^n \backslash B$. Then $g \in W^{s,\frac{n}{s}}(\R^n \backslash B)$, again by Lemma~\ref{la:trace}. Also, since $f$ is continuous on $\partial B$, $g$ is also continuous.
Set 
\[
 h := \begin{cases}
         g \quad &\text{in }\R^n \backslash B,\\
         f \quad &\text{in }B.\\
        \end{cases}
\]
By \Cref{la:glue}, $h \in W^{s,\frac{n}{s}}(\R^n)$ and $h$ is locally uniformly continuous on $\R^n \backslash B$. This last fact implies that
\[
 h_k(x) := h\brac{\frac{k+1}{k}(x)}
\]
converges uniformly to $h$ on $\partial B$ as $k \to \infty$, and also in $W^{s,\frac{n}{s}}_{loc}(\R^n)$.

Now let us consider the standard mollification $f_\eps := h_k \ast \eta_\eps$. For $\eps$ small enough in comparison with $\frac{1}{k}$, $f_\eps$ converges uniformly on $\partial B$  to $f$ and in $W^{s,\frac{n}{s}}(B))$. This completes the proof.   \end{proof}

\section{Degree of maps with signed Jacobians: Proof of Propositions~\ref{pr:degest}, \ref{pr:degest2}, \ref{pr:degeststrict}, \ref{pr:diamest}}\label{s:signedjac}
Here and hereafter, without further specifications, a null set is understood with respect to the Lebesgue measure $\mathcal{L}^n$. 

For continuous $f \in C^0(\partial B(r),\R^n)$ for some given ball $B(r) \subset \R^n$ and some point $p \in \R^n \backslash f(\partial B_r)$,  the degree of $f(B(r))$ around this point $p$ is simply the number of times that $f(\partial B(r))$ winds around $p$, {\it i.e.},
\[
 \deg(f,B_r,p) := \text{Brouwer degree of $\bigg(\psi:=\frac{f-p}{|f-p|}: \partial B(r) \to \S^{n-1}\bigg)$}.
\]
We can approximate $f$ by smooth functions $f_\eps: \partial B(r) \to \S^n$ which are uniformly close to $f$. Moreover, the Brouwer degree of $\psi=\frac{f-p}{|f-p|}$ is the same as that of $\psi_\eps = \frac{f_\eps-p}{|f_\eps -p|}$ for $\eps$ small enough, since maps that are uniformly close to each other  have the same Brouwer degree.  

For the smooth functions $\psi_\eps$ we can compute the Brouwer degree from an integral formula: denote by $\omega \in C^\infty(\Ep^{n-1}\R^n)$ the standard volume form on $\S^{n-1}$:
\[\omega =  \sum_{j=1}^n (-1)^{j-1} x^j\, dx^1 \wedge \ldots \wedge dx^{j-1} \wedge dx^{j+1} \wedge \ldots \wedge dx^n. \]

Then, for all $\eps$ small enough,
\[
 \deg(f,B(r),p) = \deg(f_\eps,B(r),p) = \int_{\partial B(r)} \psi_\eps^\ast(\omega).
\]
If we extend $\psi_\eps$ from a map $\partial B(r) \to \S^{n-1}$ to a map $\psi_\eps: B(r) \to \R^{n+1}$, then from Stokes' theorem we may obtain:
\[
  \deg(f,B(r),p)  =\int_{B(r)} \psi_\eps^\ast(d\omega) = C\, \int_{B(r)} \det(D\psi_\eps).
\]
In the last equation we used the fact that $d\omega = C\, dx^1\wedge\ldots \wedge dx^n$. Most of our arguments below are based on choosing a suitable extension of $\psi_\eps$.

\subsection{Monotonicity for non-negative Jacobian: Proof of \texorpdfstring{\Cref{pr:degest}, \Cref{pr:degest2}}{\ref{pr:degest},\ref{pr:degest2}}}
We only give the proof of \Cref{pr:degest}, the proof of \Cref{pr:degest2} is almost verbatim (it is the ``$r = 0$'' case).
\begin{proof}[Proof of \Cref{pr:degest}]
Recall that $f \in C^0(\partial B(r)) \cap C^0(\partial B(R))$ and that $p \in \R^n \backslash \brac{ f(\partial B(r)) \cup  f(\partial B(R))}$. We set
\[
 c := \min \left \{\dist( f(\partial B(r)),p),\ \dist( f(\partial B(R)),p)\right \}
\]
For this $c > 0$ let us take $d = d_c \in C^{1,1 }([0,\infty),(0,\infty))$ as in  Lemma~\ref{la:weirdode1}.

Let $f_\eps$ be the approximation in Lemma~\ref{la:approx} and set (for $\eps \ll 1$)
\[
 \psi_\eps := ( f_\eps - p)\, d(| f_\eps-p|).
\]
Then, by Stokes' theorem, we have 
\begin{equation}\label{eq:degdif:1}
\begin{split}
 \deg(f,B(R),p)-\deg(f,B(r),p) =& \deg(f_\eps,B(R),p)-\deg(f_\eps,B(r),p) \\
 =&\int_{\partial B(R)} \psi_\eps^\ast(\omega) - \int_{\partial B(r)} \psi_\eps^\ast(\omega)\\
 =&\int_{B(R)\backslash B(r)} \psi_\eps^\ast(d\omega)\\
  =&C\int_{B(R)\backslash B(r)} \det(D\psi_\eps).
\end{split}
 \end{equation}
Below we assume $p = 0$ for simplicity of notation. Observe that
\[
 D\psi_\eps = \brac{d(| f_\eps|) \, I_{n \times n} + \frac{d'(| f_\eps|)}{| f_\eps|}\,  f_\eps \otimes  f_\eps} D f_\eps;
\]
so we have $D\psi_\eps = W(f_\eps) D f_\eps$, where
\[
 W(v) := \brac{d(|v|) \, I_{n \times n} + \frac{d'(|v|)}{|v|}\, v\otimes v}.
\]
From the properties of $d$ (see \Cref{la:weirdode1}), in particular, since $d'(|v|) = 0$ whenever $|v|$ is small and $|d'(|v|)| \aeq |v|^{-2}$ whenever $|v|$ is large, we have
\begin{enumerate}
 \item $\sup_{v \in \R^n} |W(v)| < \infty$.
 \item $W \in \lip(\R^n)$.
 \item $\det(W(v)) \geq 0$ for all $v \in \R^n$. 
 
 Indeed, if $v = 0$, then $d'(|v|)= 0$ and hence $\det(W(v)) = \det(d(0) I) \geq 0$. Assume now $v \neq 0$. Since $W(v)$ is symmetric, it suffices to show that the eigenvalues are nonnegative. Observe that $v/|v|$ and any orthonormal basis of $v^\perp$ are the eigenvectors of $W(v)$. 
 
 In the former case, we compute
 \[
  W(v)\frac{v}{|v|} = d(|v|) \frac{v}{|v|} + d'(|v|) |v| \frac{v}{|v|} = \underbrace{\brac{d(|v|) + d'(|v|) |v|}}_{\geq 0} \frac{v}{|v|}.
 \]
That is, the eigenvalue for the eigenvector $v/|v|$ is non-negative. 

In the latter case, given any $o \in v^\perp$ with $|o| =1$, one has
 \[
  W(v) o = d(|v|) o + 0.
 \]
So the eigenvalue for any eigenvector $o$ perpendicular to $v$ is $d(|v|) \geq 0$.

Therefore, all the eigenvalues of $W(v)$ are non-negative, thus $\det(W(v)) \geq 0$.

\item $\det(W(v)) = 0$ whenever $|v| > \frac{c}{2}$. 

Indeed, this is because 
\[
 W(v)v = \brac{d(|v|) + d'(|v|) |v|} v
\]
where, for $|v| > \frac{c}{2}$,  we have \[
 d(|v|) + d'(|v|) |v| = \frac{1}{|v|} - \frac{1}{|v|^2} |v| = 0.
    \]
This shows that $W(v)$ is non-invertible, so $\det(W(v)) = 0$.
 \end{enumerate}
 
Now, since $W$ is Lipschitz (and globally bounded), $\det(W(f_\eps))$ lies uniformly in $W^{s,\frac{n}{s}}(B(R)\backslash B(r))$ and converges strongly in $W^{s,\frac{n}{s}}(B(R)\backslash B(r))$ to $\det(W(f))$.

On the other hand, for any fixed small enough $\eps > 0$ there exists a neighborhood of $\partial (B(R) \backslash B(r))$ where $|f_\eps - p| > \frac{3c}{4}$ -- this holds since $|f_\eps - p| > \frac{4c}{5}$ on $\partial (B(R) \backslash B(r))$ for all $\eps$ small enough, due to uniform convergence.

That is, for any small $\eps > 0$ there exists a neighborhood around $\partial (B(R) \backslash B(r))$ where $\det(W(f_\eps)) \equiv 0$. 

This implies that $\det(W(f_\eps))$ and $\det(W(f))$ all lie in $W^{s,\frac{n}{s}}_0(B(R) \backslash B(r))$. By virtue of Lemma~\ref{la:glue}, we can extend these functions by zero to all of $\Omega$, and they belong consequently to $W^{s,\frac{n}{s}}_0(\Omega)$.

That is, we have shown that 
\[
 \deg(f,B(R),p) - \deg(f,B(r),p) = \lim_{\eps \to 0} \int_{\Omega} \det(Df_\eps)\, \det(W(f_\eps)) = \Jac(f)[\det(W(f))].
\]
The right-hand side is nonnegative by assumption, and \Cref{pr:degest} is proven.
\end{proof}

\subsection{Positive Jacobian implies sense-preserving: Proof of \texorpdfstring{\Cref{pr:degeststrict}}{\ref{pr:degeststrict}}}
\begin{proof}[Proof of \Cref{pr:degeststrict}]
The proof is very similar to that of \Cref{pr:degest}. With the notation used therein, we have
\[
  \deg(f,B(r),p)=\Jac(f)[\det(W(f))],
\]
where again
\[
 W(v) := \brac{d(|v|) \, I_{n \times n} + \frac{d'(|v|)}{|v|}\, v\otimes v}.
\]
for $d$ taken from Lemma~\ref{la:weirdode1} with $c := \frac{1}{2} \dist(f(\partial B(r)),p)$.

As before, we have $\det(W(f)) \geq 0$. The assumption $\Jac(f) > 0$ implies  $\deg(f,B(r),p) \geq 0$.  It remains to show that if $\deg(f,B(r),p) = 0$ then $p \not \in f(B(r))$, for the claim that $\deg(f,B(r),p) \geq 1$ in $f(B(r))$ shall follow immediately.

So, assume that $p \not \in f(\partial B(r))$ and $\deg(f,B(r),p) = 0$. From Definition~\ref{def:jac} we see that $\Jac(f) > 0$ readily implies $\det(W(f)) \equiv 0$. That is, one of the eigenvalues of $W(f)$ is zero. As computed in the proof of Proposition~\ref{pr:degest}, the eigenvalues of $W(v)$ are 
\[
\brac{d(|v|) + d'(|v|) |v|} \quad \mbox{and} \quad d(|v|).
\]
Since $d(|v|) \neq 0$ for all $v$, $\det(W(f(x))) \equiv 0$ implies that necessarily
\[
d(|f(x)-p|) + d'(|f(x)-p|) |f(x)-p|  = 0 \quad \mbox{for all $x \in B(r)$}.
\]
By the properties of $d$ (see Lemma~\ref{la:weirdode1}), we deduce that $\inf_{B(r)} |f(x)-p| > 0$. Thus $p \not \in f(B(r))$ as claimed.
\end{proof}

\subsection{Comparability of diameters: Proof of \texorpdfstring{\Cref{pr:diamest}}{\ref{pr:diamest}}}
The proof below is an adaptation from the argument in \cite{S88,GHP17}. Modifications are necessary due to the fact that we do not have a pointwise Jacobian.

\begin{proof}
Since $f$ is continuous on $\partial B(r)$, we can find a large ball $B(q,\rho)$ of radius $\rho := \diam f(\partial B(r))$ such that $f(\partial B(r)) \subset B(q,\rho)$.

Take $\pi=\pi_\lambda$ from Lemma~\ref{la:weirdode2} for $\lambda := 10 \rho$.

Let $f_\eps$ be the smooth approximation of $f$ from Lemma~\ref{la:approx}. For all small enough $\eps > 0$ we have $f_\eps(\partial B(r)) \subset B(q,2\rho)$.

In particular if we set \[g_\eps := (f_\eps-q)\, \pi(|f_\eps-q|) +q\] then $g_\eps = f_\eps$ on $\partial B(r)$. Consequently (by an integration by parts argument it is easy to see that the integral of the Jacobian of a map on a ball only depends on the boundary value of that map, \cite[Lemma 4.7.2]{IM01}),
\[
 \int_{B(r)} \det(Df_\eps) = \int_{B(r)} \det(Dg_\eps).
\]
Computing $Dg_\eps$ similar as in the proof of Proposition~\ref{pr:degest}, setting
\[
 W(v) = \pi(|v|) I_{n \times n} + \frac{\pi'(|v|)}{|v|} v \otimes v,
\]
we obtain
\[
 \int_{B(r)} \det(Df_\eps) \Big ( 1 - \det(W(f_\eps-q)) \Big ) = 0.
\]
As in the proof of Proposition~\ref{pr:degest}, the map $1 - \det(W(f_\eps-q))$ belongs to $W^{s,\frac{n}{s}}(B(r))$ and converges strongly in that space to $1 - \det(W(f-q))$.

Moreover, as in the proof of Proposition~\ref{pr:degest}, we can compute 
\[
\det(W(v)) = \pi(|v|)^{n-1}\, \big (\pi(|v|) + |v| \pi'(|v|)  \big ),
\]
and by the properties of $\pi$, see Lemma~\ref{la:weirdode2},
\[
 1 - \det(W(f_\eps-q))  \geq 0 \quad \mbox{a.e. in $B(r)$}.
\]
Moreover, since $\pi(|v|) \equiv 1$ for $|v| \leq 10\rho$, we have 
\[
 W(f_\eps-q) \equiv I_{n \times n} \quad \text{close to $\partial B(r)$}.
\]
That is, 
\[
 1 - \det(W(f_\eps-q)) \equiv 0 \quad \text{close to $\partial B(r)$}.
\]
By Lemma~\ref{la:trace} and Lemma~\ref{la:glue} we can thus again extend $1 - \det(W(f_\eps-q))$ and $1 - \det(W(f-q))$ by zero to a $W^{s,\frac{n}{s}}_0(\Omega)$-function. Thus, we conclude that 
\[
 \Jac(f)\left [1 - \det(W(f-q))\right ] = 0.
\]

Since by assumption $\Jac(f) > 0$ and $1 - \det(W(f-q)) \geq 0$ a.e., we may infer (see \Cref{def:jac}) that 
\[
 1 - \det(W(f-q)) \equiv 0.
\]
That is,
\[
 \pi(|f-q|)^{n-1}\, \big (\pi(|f-q|) + |f-q| \pi'(f-q)  \big ) \equiv 1
\]
But by the properties of $\pi$, see Lemma~\ref{la:weirdode2}, this implies
\[
 |f(x)-q| < 2\lambda = 20\rho = 20\, \diam(f(\partial B(r))) \quad \mbox{a.e. in $x \in B(r)$}.
\]
Therefore,
\[
 \left \{x \in B(r):\ |f(x)-q| \geq 20\, \diam(f(\partial B(r))) \right \} \quad \text{is a null set}.
\]
\end{proof}

\section{Continuity of maps with positive Jacobian: Proof of Theorem~\ref{th:main}, \ref{th:main2}}\label{s:mainproof}
The proof of \Cref{th:main} crucially relies on the diameter estimates of \Cref{pr:diamest}. Once we have this, we adapt the argument in \cite{S88} to fractional Sobolev spaces in a more or less straightforward fashion, namely 
Theorem~\ref{th:main} is a corollary of the following statement.
\begin{proposition}\label{pr:osciscontinuity}
Let $\Omega \subset \R^n$. Assume that $f \in W^{s,\frac{n}{s}}(\Omega,\R^n)$, $s \in (0,1)$ satisfies the following: for any $x_0 \in \Omega$ and $\mathcal{L}^1$-almost all radii $0 < r < \rho < \dist(x_0,\Omega)$, there holds
\[
  \osc_{\partial B(x_0,r)} f \leq \osc_{\partial B(x_0,\rho)} f.
\]
Then $f$ is continuous. Moreover, for any ball $B \Subset \Omega$, $s > 0$, and $x,y \in B$, we have 
\[
 |f(z)-f(y)|^{p} \leq \frac{1}{C(s,p,B) - \log(|x-y|)} [f]_{W^{s,\frac{n}{s}}(B)}^{\frac{n}{s}}.
\]
\end{proposition}
Observe that an easy extension of Proposition~\ref{pr:osciscontinuity} holds for $W^{s,p}$-maps whenever $s-\frac{n-1}{p} > 0$.

\begin{proof}[Proof of Theorem~\ref{th:main}]
Fix $x_0 \in \Omega$ and let $R := \dist(x_0,\partial \Omega)$. W.l.o.g. $x_0 = 0$. Let $0 < r < \rho < R$ such that $f$ is continuous on $\partial B(r)$ and $\partial B(\rho)$. By Lemma~\ref{la:restriction} we know this happens for $\mathcal{L}^1$-a.e. $r$ and $\rho$.

By \Cref{pr:diamest}, for almost any $0 < r < \rho < R$ we have the monotonicity
\begin{equation}\label{eq:diammon}
 \diam(f(\partial B(r))) \leq \Lambda\, \diam(f(\partial B(\rho))).
\end{equation}
Indeed, by Lemma~\ref{la:restriction}, for almost any $0 < r < \rho < R$ the map $f$ is continuous on $\partial B(r)$ and $\partial B(\rho)$. Thus
\[
 \diam(f(B(\rho))) \overset{\text{P.\ref{pr:diamest}}}{\leq} \Lambda\, \diam(f(\partial B(\rho))),
\]
where $\diam(f(B(\rho))$ is understood as 
\[
 \diam(f(B(\rho)) = \inf \left \{\diam(f(A)):\quad A \subset B(\rho),\ |B(\rho)\backslash A| = 0 \right\}. 
\]
The trace $f\Big |_{\partial B(r)}$ is $\mathcal{H}^{n-1}$-a.e. attained by sequences of $f\Big |_{\partial B(\tilde{r})}$ as $\tilde{r} \to r$. If $A$ is as in the definition of $\diam$ above then for $\mathcal{L}^1$-almost every $\tilde{r}$ we have $\mathcal{H}^{n-1}(A \cap \partial B(\tilde{r})) =  \mathcal{H}^{n-1}(\partial B(\tilde{r}))$. So, we find a sequence $r_i \to r$ with $\partial \mathcal{H}^{n-1}(A \cap \partial B(r_i)) =  \mathcal{H}^{n-1}(\partial B(r_i))$ and $f\Big |_{\partial B(r_i)}$ converging $\mathcal{H}^{n-1}$-a.e. to $f\Big |_{\partial B(r)}$. Thus, whenever $\rho > r$,
\[
 \diam(f(\partial B(r))) \leq \diam(f(B(\rho))). 
\]
This establishes \eqref{eq:diammon}.

Now, we may deduce from \eqref{eq:diammon} that, for almost any $0 < r < \rho < R$, 
\[
 \osc_{\partial B(r)} f \leq \Lambda\, \osc_{\partial B(\rho)} f.
\]
From here one concludes the continuity property with \Cref{pr:osciscontinuity}.
% 
% As for openness, take a ball $B \subset \subset \Omega$, and let $p \in f(B)$. That is, there exists $x_0 \in B$ such that $f(x_0) = p$. Let $\eps$ be so small such that $B(x_0,\eps) \subset \subset B$. Then either $p \in f(\partial B(x_0,\eps)) \subset f(B)$ or $p \in f(B(x_0,\eps)) \backslash f(\partial B(x_0,\eps))$.
% In the former case \ToDo(???), in the later case, since $\Jac(f) > 0$ we have by Proposition~\ref{pr:degeststrict} that $\deg(p,f,B(x_0,\eps)) \geq 1$. The degree $p \mapsto \deg(p,f,B(x_0,\eps))$ of a continuous function $f$ is continuous, and thus for all $q \approx p$ we have $\deg(q,f,B(x_0,\eps)) > 0$, that is $q \in f(B(x_0,\eps))$ for all $q \approx p$. Thus $f(B)$ is an open set.
\end{proof}

\begin{proof}[Proof of Proposition~\ref{pr:osciscontinuity}]
By Sobolev embedding for $s > 0$,
\[
 \osc_{\partial B(\rho)} f \aleq \rho^{\frac{s}{n}} [f]_{W^{s,\frac{n}{s}}(\partial B(\rho))}.
\]
Thus we find
\[
 \brac{\osc_{\partial B(x,r)} f}^{\frac{n}{s}} \log(R/r) \leq \int_r^R \frac{1}{\rho} \brac{\osc_{\partial B(\rho)}f}^{\frac{n}{s}} d\rho \leq \int_{r}^R [f]_{W^{s,\frac{n}{s}}(\partial B(\rho))}^{\frac{n}{s}}\, d\rho
\]
In view of Lemma~\ref{la:restriction} we obtain 
\[
 \brac{\osc_{\partial B(x,r)} f}^{\frac{n}{s}} \leq \frac{1}{\log(R/r)} [f]_{W^{s,\frac{n}{s}}(B(R))}^{\frac{n}{s}}.
\]
This readily implies the claim.
\end{proof}
Theorem~\ref{th:main2} also follows from Proposition~\ref{pr:osciscontinuity}, and additionally the following distortion argument.
\begin{lemma}\label{la:curlfreedist}
Let $f \in W^{s,\frac{2}{s}}(\Omega,\R^2)$, $\Omega \subset \R^2$ open, $s \geq \frac{2}{3}$.
Assume that $\Jac(f) \geq 0$ and $\curl(f) := \partial_1 f^2 - \partial_2 f^1 = 0$ in distributional sense. 

For $\delta \in \R \backslash \{0\}$ set $f_\delta(x_1,x_2) := f(x_1,x_2) + \delta (-x_2,x_1)^T$. 
Then $\Jac(f_\delta) > 0$.
\end{lemma}
\begin{proof}
Let $f_\eps$ be an approximation of $f$ in $W^{s,\frac{2}{s}}(\Omega)$. We have
\[
 \det(Df_\eps + \delta (-x_2,x_1)^T) = \det(Df_\eps) + \delta^2 + \delta(\partial_1 f^2_\eps -\partial_2 f^1_\eps).
\]
That is, for any $\varphi \in C_c^\infty(\Omega)$,
\[
\begin{split}
 \Jac(f_\delta)[\varphi] =& \lim_{\eps \to 0} \int_{\Omega} \det(Df_\eps + \delta (-x_2,x_1)^T)\varphi\\
 =& \Jac(f)[\varphi] + \delta^2 \int \varphi + \delta \lim_{\eps \to 0} \int_{\Omega} \brac{\partial_1 f^2_\eps -\partial_2 f^1_\eps}\, \varphi.
\end{split}
 \]
Integrating by parts and the pointwise a.e. convergence of $f_\eps$ to $f$ implies
\[
 \int_{\Omega} \brac{\partial_1 f^2_\eps -\partial_2 f^1_\eps}\, \varphi = -\curl[f_\eps](\varphi) \xrightarrow{\eps \to 0} \curl[f](\varphi) = 0.
\]
Thus,
\[
  \Jac(f_\delta)[\varphi] = \Jac(f)[\varphi] + \delta^2 \int \varphi.
\]
In particular if $\varphi \geq 0$ and $\Jac(f_\delta)[\varphi] = 0$ we have $\varphi \equiv 0$, i.e. $\Jac(f_\delta) > 0$.
\end{proof}

\begin{proof}[Proof of Theorem~\ref{th:main2}]
By Lemma~\ref{la:curlfreedist} and Theorem~\ref{th:main} we have $f_\delta$ is continuous, and indeed we have a estimate on the modulus of continuity of $f_\delta$ by Proposition~\ref{pr:osciscontinuity}. This estimate is uniform in $\delta$, and by Arzela-Ascoli we conclude that $f = \lim_{\delta \to 0} f_\delta$ still enjoys the same continuity estimate.
\end{proof}

\appendix
\section{Two functions}
\begin{lemma}\label{la:weirdode1}
For any $c > 0$ there exists $d = d_c \in C^{1,1}([0,\infty),(0,\infty))$ such that
 \[
 \begin{cases} 
  d(t) + td'(t) \geq 0 \quad &\forall t > 0\\
 d(t) = t^{-1} \quad &\text{for $t > c/2$}\\
 d(t) \equiv d(0)\quad& \text{for $t \approx 0$}\\
 d(t) + td'(t) > 0 \quad& \text{for $t \approx 0$}.
 \end{cases}
\]
\end{lemma}
\begin{proof}
Observe that if $d(t)$ satisfies the above assumptions for $c =2$, then $d_c(t) := c^{-1} d(t/c)$ satisfies the assumptions for generic $c>0$. So, w.l.o.g. $c = 2$.

Set
\[
 d(t) := \begin{cases}
            t^{-1} \quad &\text{$t \geq 1$}\\
            -t^2+t+1\quad &t \in [\frac{1}{2},1]\\
            \frac{5}{4}\quad &t \in [0,\frac{1}{2}].
         \end{cases}
\]
Observe that $\lim_{t \to 1^+} d(t)= \lim_{t \to 1^-} d(t) = 1$, and that $\lim_{t \to 1/2^+} d(t)= \lim_{t \to 1/2^-} d(t) = \frac{5}{4}$. Also,
\[
 d'(t) := \begin{cases}
            -t^{-2} \quad &\text{$t \geq 1$}\\
            -2t+1\quad &t \in [\frac{1}{2},1]\\
            0\quad &t \in [0,\frac{1}{2}].
         \end{cases}
\]
In particular, $\lim_{t \to 1^+} d'(t)= \lim_{t \to 1^-} d'(t) = -1$ and $\lim_{t \to 1/2^+} d'(t)= \lim_{t \to 1/2^-} d'(t) = 0$.
That is, $d \in C^{1,1}([0,\infty))$. The only thing left to check is that
\[
 d(t) + t d'(t) = \begin{cases}
                   0 \quad &t \geq 1\\
                     -3t^2+2t+1&t \in [\frac{1}{2},1]\\
                    \frac{5}{4}\quad &t \in [0,\frac{1}{2}]
                  \end{cases}
\]
is non-negative. But this is immediate.  \end{proof}

\begin{lemma}\label{la:weirdode2}
For any $n \in \N$ and any $\lambda > 0$ there exists $\pi \in C^{1,1}([0,\infty),(0,\infty))$ with the following properties:
 \[
 \begin{cases} 
  \pi(t)^{n-1}(\pi(t) + t\pi'(t)) \leq 1 \quad &\forall t \geq 0\\
  \pi(t)^{n-1}(\pi(t) + t\pi'(t)) < 1 \quad &\forall t \geq 2\lambda\\
  \pi(t) \equiv 1\quad& \text{for $t \leq \lambda$}\\
  \sup_{t} \pi(t) < C\quad &\text{with $C$ independent of $\lambda$}\\
  \sup_{t} \pi'(t) < C(\lambda).
 \end{cases}
\]
\end{lemma}
\begin{proof}
Setting $\pi_\lambda(t) := \pi(t/\lambda)$ we can reduce to the case $\lambda = 1$, which we shall now consider.

Set $r(t) := t^n \pi(t)^n$. Then the differential inequalities become
\[
 \begin{cases} 
  r'(t) \leq n t^{n-1} \quad &\forall t \geq 0\\
  r'(t) < nt^{n-1} \quad &\forall t \geq 2\\
  r(t) = t^n \quad& \text{for $t \leq 1$}\\
 \end{cases}
\]
Clearly we can find a $C^{1,1}$-function $r$ that satisfies these conditions, {\it e.g.}
\[
 r(t) := 
 \begin{cases}
 r(t) = t^n \quad& \text{for $t \leq 1$}\\
 r(t) = t^n-a(t) \quad& \text{for $1 < t < 2$}\\
 r(t) = t^n - \frac{1}{2}t \quad & \text{for $t \geq 2$}
 \end{cases}
\]
where $a(t)$ is any smooth non-decreasing function such that $a(1) = a'(1) = 0$, $a(2) = 1$, and $a'(2) = \frac{1}{2}$.

Then $p(t) := t^{-1}\sqrt[n]{r(t)}$ is bounded, has derivatives bounded, and satisfies all the other assumptions as well.
\end{proof}

\bibliographystyle{abbrv}%
\bibliography{bib}%

\end{document}